\documentclass[11pt,a4paper]{amsart}
\usepackage{amssymb,amsmath}
\usepackage{graphicx}
\numberwithin{equation}{section}
\vspace{8cm}
\newtheorem{theorem}{Theorem}[section]
\newtheorem{proposition}[theorem]{Proposition}
\newtheorem{corollary}[theorem]{Corollary}

\newtheorem{remark}[theorem]{Remark}

\newcommand{\R}{{\mathbb R}}

\newcommand{\refe}[1]{{(\ref{#1})}}
\newcommand{\dis}{\displaystyle}

\newcommand{\noi}{\noindent}

\newcommand{\Mp}{\mathcal{M}^+_{\lambda, \Lambda}}
\newcommand{\uog}{u^\omega_\gamma}
\newcommand{\Oog}{\Omega^\omega_\gamma}

\newcommand{\Ooga}{\Omega^\omega_{\gamma, a}}
\newcommand{\uoga}{u^\omega_{\gamma, a}}

\begin{document}
\parindent=0pt

\title[Symmetry minimizes the principal eigenvalue]{Symmetry minimizes the principal eigenvalue: an example for the Pucci's sup operator}

\author[ I. Birindelli, F. Leoni]
{ Isabeau Birindelli and Fabiana Leoni}

\address{Dipartimento di Matematica\newline
\indent Sapienza Universit\`a  di Roma \newline
 \indent   P.le Aldo  Moro 2, I--00185 Roma, Italy.}
 \email{isabeau@mat.uniroma1.it}
\email{leoni@mat.uniroma1.it}

\begin{abstract}
We explicitly evaluate the principal eigenvalue of the extremal Pucci's sup--operator for a class of  plane domains, and we prove that, for fixed area, the eigenvalue is minimal for the most symmetric set.
 \end{abstract}
\maketitle

\section{Introduction}\label{intro}
In 1951, P\'olya and Szego conjectured:

{\em Of all $n$-polygons with the same area, the regular $n$-polygon has the smallest first Dirichlet eigenvalue,}

referring to the Dirichlet eigenvalue of the Laplacian. It is very simple to see 
that among all rectangles of same area, the one that minimizes the first Laplace Dirichlet 
eigenvalue is the square.  Using Steiner symmetrization, P\'olya and Szego proved the conjecture for $n=2$ and $n=3$,  
but it is still an open problem for $n>4$. On the other hand, the well known Faber-Krahn's inequality, affirms  that in any dimension,
among  all domains of same volume, the euclidean ball has the smallest first Laplace Dirichlet eigenvalue.

The notion of the first Dirichlet eigenvalue for linear elliptic operators, has been extended to fully nonlinear ones (see \cite{BD,BEQ,IY}). 
Indeed, for linear operators, Berestycki, Nirenberg and Varadhan in \cite{BNV} 
use the maximum principle to define the principal eigenvalue. So, following their idea it is possible to prove that,
if $\Mp$ denotes the Pucci's  supremum operator, with ellipticity constants $0<\lambda\leq\Lambda$ and 
if $\Omega$ is a bounded Lipschitz domain, then there exists $\phi>0$ in $\Omega$ such that
$$
\left\{\begin{array}{lc}
  \Mp(D^2 \phi)+\mu^+(\Omega) \phi = 0 & \mbox{in}\ \Omega\\
  \phi=0& \mbox{on}\ \partial\Omega
\end{array}
\right.
$$
for 
$$\mu^+(\Omega)=\sup\{\mu\in\R\, :\  \exists\ \phi>0 \ \mbox{in}\ \Omega,\  \Mp(D^2 \phi)+\mu \phi \leq 0\quad\mbox{in}\ \Omega\}.$$
For 
$$\mu^-(\Omega)=\sup\{\mu\in\R\, :\ \exists\ \phi<0 \ \mbox{in}\ \Omega,\  \Mp(D^2 \phi)+\mu \phi \geq 0\quad\mbox{in}\ \Omega\}$$
the existence of a negative eigenfunction is similarly proved.

It is hence quite natural, to wonder if the Faber-Krahn inequality is valid for these "eigenvalues" associated 
to $\Mp$; precisely, given a ball $B$ is it true that 
\begin{equation}\label{FK}
\mu^+(B)\leq\mu^+(\Omega), \ \mbox{for any}\ \Omega \ \mbox{such that}\  |\Omega|=|B|?  
\end{equation}
Here $|.|$ indicates the volume.

Faber-Krahn inequality is proved in several ways, the most classical one uses Steiner symmetrization 
together with the Rayleigh quotient that defines the eigenvalue. Clearly these tools are not at all adapted to this non variational
fully nonlinear setting.
Another possible proof relies on a more geometrical understanding of the problem;
as it is well-explained in \cite{SI},  a domain $\Omega$ is critical  for the Laplace first eigenvalue functional  under fixed volume variation, if and only if  the eigenfunction $\phi>0$ associated to $\mu(\Omega)$ has 
constant Neumann boundary condition i.e. if it is a solution of an 
overdetermined boundary value problem. This is proved using Hadamard's identity (we refer to \cite{SI} and references therein). But, by
Serrin's classical result, the only bounded domains which admit non trivial solutions satisfying overdetermined boundary conditions
are balls.
In \cite{BDov}, it is proved that at least for $\lambda$ and $\Lambda$ close enough, the only bounded domains for which 
the overdetermined boundary value problem associated to $\Mp$ admits a non trivial solution are the balls. 
This suggests that (\ref{FK}) may be true.
Unfortunately,  it is not known if, for  the eigenvalue functional associated to $\Mp$, the critical domains under fixed volume have  
eigenfunctions with constant normal derivative.

Both the Faber-Krahn inequality and the P\'olya and Szego conjecture state that symmetry of the domain decreases 
the Laplace first eigenvalue. 
If this is true for the Pucci eigenvalue  is not known but the scope of this paper is to show that among a family of subsets of 
$\R^2$ of same area, which are in some sense deformations of rectangles, 
the one that minimizes $\mu^+(\cdot)$ is the most symmetric one. This minimal domain will be denoted 
$\Omega^\omega_1$ for $\omega=\frac{\Lambda}{\lambda}$ and it is, somehow, a deformation of a square. 
The result is accomplished by explicitly computing the eigenvalue $\mu^+(\Omega^\omega_1)$ and the corresponding eigenfunction. 
Observe that the square is not the good set to consider, since, as it 
is proved in Proposition \ref{rect}, the eigenfunction associated to the square is not the product of 
two functions of one variable.

Remarkably, an analogous explicit computation of $\mu^-(\cdot)$  leads to  unbounded sets. In particular,
one can construct a symmetric unbounded set $D^\omega_1$ such that $\mu^-(D^\omega_1)=\lambda$.

\section{The principal eigenvalue of $\Mp$ in some special domains}
In order to fix notations, we recall that the supremum Pucci operator is defined by
$$\dis \Mp(X)=\lambda \sum_{e_i<0} e_i+\Lambda \sum_{e_i>0} e_i$$
where $e_i$ are the eigenvalues of the symmetric matrix $X$ and $\Lambda \geq \lambda>0$ are fixed constants.
The starting point of our analysis is the following observation.

\begin{proposition}\label{rect}
For $\Lambda >\lambda>0$, any eigenfunction of $\Mp$ associated with the positive principal eigenvalue in any squared  domain $Q\subset \R^2$ is not a function of separable variables.
\end{proposition}

\begin{proof} Let $Q=\left( -\frac{\pi}{\sqrt{2}},\frac{\pi}{\sqrt{2}}\right)^2$ and let $u(x,y)$ be the principal eigenfunction of $\Mp$ in $Q$ associated with the positive principal eigenvalue $\mu>0$, that is
\begin{equation}\label{eigen}
\left\{ \begin{array}{l}
-\Mp (D^2u)= \mu \, u\quad \hbox{in }\ Q\, ,\\[2ex]
u>0 \quad \hbox{in }  Q\, ,\ u=0 \quad \hbox{on } \partial Q\,.
\end{array} \right.
\end{equation}
Assume, by contradiction, that $u$ is a function of separable variables. Then, by symmetry and regularity results, $u$ can be written as 
$$
u(x,y)= f(x)\, f(y)
$$
with $f:\left( -\frac{\pi}{\sqrt{2}},\frac{\pi}{\sqrt{2}}\right)\to \R$ smooth, positive, even, and, up to a normalization, satisfying  $f(0)=1$.  In particular, one has
$$
D^2u(0,y)=\left( \begin{array}{cc}
f^{''}(0)f(y) & 0\\
0 & f^{''}(y) \end{array}\right)
$$
and equation \refe{eigen} tested at $(0,0)$ yields
$$
f^{''}(0)=- \frac{\mu}{2\lambda } <0\, .
$$
Moreover, if  for some  $y_0\in \left( -\frac{\pi}{\sqrt{2}},\frac{\pi}{\sqrt{2}}\right)$ one has $f^{''}(y_0)=0$, then from equation \refe{eigen} written for $(x,y)=(0,y_0)$ we obtain the contradiction
$$
-\lambda\, f^{''}(0) = \mu =-2 \lambda\, f^{''}(0)\, .
$$
Therefore, we have $f^{''}<0$ in $\left(- \frac{\pi}{\sqrt{2}},\frac{\pi}{\sqrt{2}}\right)$ and, again from equation \refe{eigen}, we deduce that $f$ satisfies
$$
\left\{
\begin{array}{l}
f^{''}=-\frac{\mu}{2 \lambda} \, f\, ,\quad f>0 \quad \hbox{in } \left( -\frac{\pi}{\sqrt{2}},\frac{\pi}{\sqrt{2}}\right)\\[2ex]
f\left(-\frac{\pi}{\sqrt{2}}\right)=f\left(\frac{\pi}{\sqrt{2}}\right)=0\, ,\ f(0)=1
\end{array}
\right.
$$
Hence, $\mu=  \lambda$ and $f(x)=\cos \left(\frac{x}{\sqrt{2}}\right)$. On the other hand, for the function $u(x,y)=\cos \left(\frac{x}{\sqrt{2}}\right)\, \cos \left(\frac{y}{\sqrt{2}}\right)$ one has, in particular,
$$
D^2u (x,x)=
\frac{1}{2}\left( \begin{array}{cc}
-\cos^2\left(\frac{x}{\sqrt{2}}\right) &  \sin^2\left(\frac{x}{\sqrt{2}}\right) \\
\sin^2\left(\frac{x}{\sqrt{2}}\right) & -\cos^2\left(\frac{x}{\sqrt{2}}\right) \end{array}\right)\, ,
$$
and, for $\frac{\pi}{2\sqrt{2}}\leq |x|< \frac{\pi}{\sqrt{2}}$ we have
$$
-\Mp(D^2u (x,x))= \Lambda \cos^2\left(\frac{x}{\sqrt{2}}\right) +\frac{\Lambda -\lambda}{2} \neq  \lambda \, u(x,x)\, ,
$$
unless $\Lambda =\lambda$.
\end{proof}

Let us remark that the function 
$$
u(x,y)=\cos \left(\frac{x}{\sqrt{2}}\right)\, \cos \left(\frac{y}{\sqrt{2}}\right)=\frac{1}{2} \left[ \cos\left( \frac{x+y}{\sqrt{2}}\right) +\cos\left( \frac{x-y}{\sqrt{2}}\right)\right]
$$
 is an eigenfunction for the Laplace operator in the squared domain $Q$ relative to the first eigenvalue $\lambda_1(-\Delta, Q)=1$. As long as $u$ is concave, it also satisfies the equation
$$
-\Mp (D^2u)=-\lambda \, \Delta u= \lambda\, u\, .
$$
Actually this is the case for $(x,y)\in Q_1=\{ (x,y)\in \R^2\, : \, |x|+|y|< \frac{\pi}{\sqrt{2}}\}$, the rotated squared domain with side $\pi$. Moreover, the same holds true  for any function of the form
$$
u_\gamma (x,y)= \gamma\,   \cos\left( \frac{x+y}{\sqrt{2}}\right) +\cos\left( \frac{x-y}{\sqrt{2}}\right)\, ,
$$
with $\gamma>0$. In the next result we suitably extend the function $u_\gamma |_{Q_1}$ in order to obtain an eigenfunction for $\Mp$ relative to the eigenvalue $\lambda$.
\medskip

Let $\omega \geq 1$ be a parameter to be fixed in the sequel, and, for $\frac{1}{\sqrt{\omega}}\leq \gamma\leq \sqrt{\omega}$ let us introduce the positive even functions defined for $|x|\leq \frac{\pi}{2}+\sqrt{\omega} \arcsin \left(\frac{1}{\gamma \sqrt{\omega}}\right)$ as
$$
\phi^\omega_\gamma (x)=\left\{ 
\begin{array}{ll}
\frac{\pi}{2} +  \sqrt{\omega} \arcsin \left(\frac{\gamma}{\sqrt{\omega}}\cos x\right)  &  \hbox{if } |x|\leq \frac{\pi}{2}\\[2ex]
\arccos \left(\gamma \sqrt{\omega} \sin \left( \frac{|x|-\pi/2}{\sqrt{\omega}} \right)\right) & \hbox{if }  \frac{\pi}{2}<|x|\leq \frac{\pi}{2}+\sqrt{\omega} \arcsin \left(\frac{1}{\gamma \sqrt{\omega}}\right)
\end{array}
\right.
$$
Note that
\begin{equation}\label{fiog}
\phi^\omega_{\gamma^{-1}}= \left(\phi^\omega_\gamma\right)^{-1}
\end{equation}
so that, in particular,  $\phi^\omega_1=\left(\phi^\omega_1\right)^{-1}$.

\begin{figure}
\includegraphics[height=40mm]{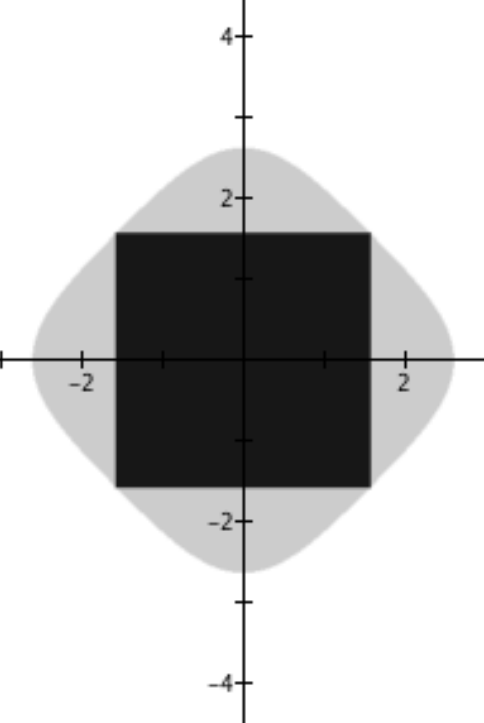}
\includegraphics[height=40mm]{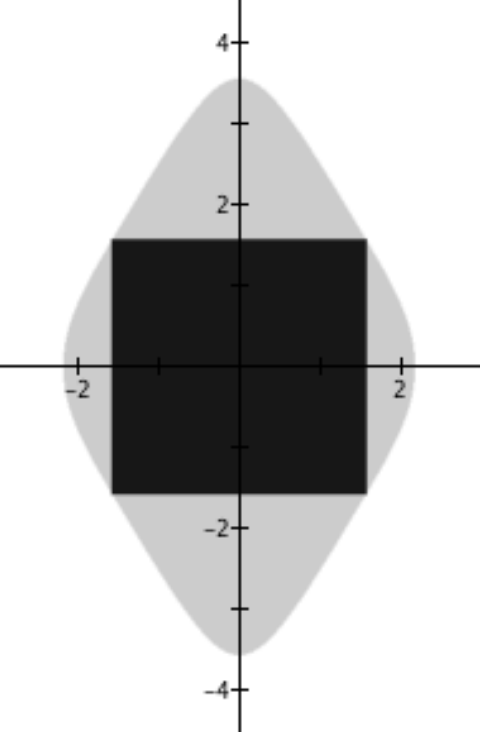}
\includegraphics[height=40mm]{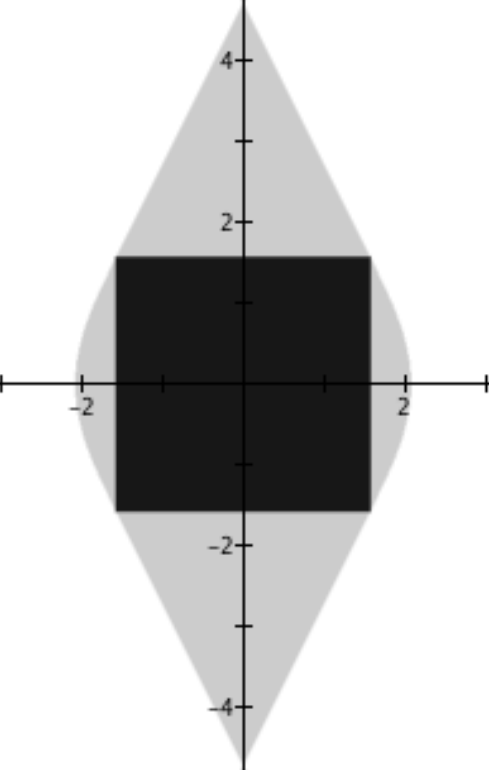}
\caption{$\Omega^\omega_1$,  $\Omega^\omega_{\gamma}$, $\Omega^\omega_{\sqrt{\omega}}$, three domains for which the eigenvalue is $\lambda$; in the black square $u_\gamma^\omega$ is concave.}
\end{figure}

Next, let us consider the open bounded subsets
$$
\Omega^\omega_\gamma \,: = \left\{ (x,y)\in \R^2\, : \, |y|< \phi^\omega_\gamma (x)\right\}\, .
$$
Note that for $\omega =1$ we have $\gamma=1$ and $\Omega^1_1$ is nothing but the rotated squared domain  with side $\sqrt{2} \pi$.
In general, $\Omega^\omega_\gamma$ is  a Lipschitz domain symmetric both with respect to the $x$ and $y$ axes, and, by \refe{fiog}, 
$$
\Omega^\omega_{\frac{1}{\gamma}}=\left\{ (x,y)\in \R^2\, :\, (y,x)\in \Omega^\omega_\gamma \right\}\, .
$$
In particular, $\Omega^\omega_1$ is symmetric also with respect to the diagonal $y=x$.

\begin{theorem} \label{simm} Given $\Lambda\geq \lambda>0$ let us set $\omega =\frac{\Lambda}{\lambda}\geq1$. Then, for any $\frac{1}{\sqrt{\omega}}\leq \gamma\leq \sqrt{\omega}$, the positive principal eigenvalue of $\Mp$ in the domain $\Omega^\omega_\gamma$ is
$$
\mu \left( \Omega^\omega_\gamma\right) = \lambda
$$
and the principal  eigenfunction is, up to positive constants,
$$
u^\omega_\gamma(x,y)=\left\{
\begin{array}{ll}
\gamma\,  \cos x +\cos y & \hbox{if } |x|\leq \frac{\pi}{2}\,,\  |y|\leq \frac{\pi}{2}\\[2ex]
\gamma \, \sqrt{\omega} \cos \left( \frac{|x|-\pi/2}{\sqrt{\omega}}+\frac{\pi}{2}\right) + \cos y    & \hbox{if } (x,y)\in \Omega^\omega_\gamma\,,\    |x|\geq \frac{\pi}{2}\\[2ex]
\gamma\,  \cos x  +\sqrt{\omega} \cos \left( \frac{|y|-\pi/2}{\sqrt{\omega}}+ \frac{\pi}{2}\right)  & \hbox{if } (x,y)\in \Omega^\omega_\gamma\,,\    |y|\geq \frac{\pi}{2}  \end{array} \right.
 $$
 \end{theorem}
 
 \begin{proof}
 The proof is a straightforward computation. We observe that $\uog$ is smooth and positive in $\Oog$, and it vanishes on $\partial \Oog$. For $|x|\leq \frac{\pi}{2}\,,\  |y|\leq \frac{\pi}{2}$ one has
 $$
 D^2\uog (x,y)= \left( 
 \begin{array}{cc}
 -\gamma \, \cos x & 0\\
 0 & -\cos y
 \end{array}\right)
 $$
 Therefore, for $|x|\leq \frac{\pi}{2}$ and  $|y|\leq \frac{\pi}{2}$,  $\uog$ is concave and it satisfies
 $$
 -\Mp \left( D^2\uog\right) =-\lambda\, \Delta \uog =\lambda\, \uog\, .
 $$
 For $(x,y)\in \Omega^\omega_\gamma$ and  $|x|\geq \frac{\pi}{2}$, one has
 $$
 D^2\uog (x,y)= \left( 
 \begin{array}{cc}
\frac{\gamma}{\sqrt{\omega}} \sin \left( \frac{|x|-\pi/2}{\sqrt{\omega}}\right)& 0\\
 0 & -\cos y
 \end{array}\right)
 $$
Note that, if $(x,y)\in \Omega^\omega_\gamma$ and  $|x|\geq \frac{\pi}{2}$, then $|y|\leq \frac{\pi}{2}$ and $0\leq \frac{|x|-\pi/2}{\sqrt{\omega}}<\arcsin \left( \frac{1}{\gamma \sqrt{\omega}}\right) \leq \frac{\pi}{2}$; therefore
$$
-\Mp \left( D^2\uog\right) =\lambda \, \cos y-\Lambda\, \frac{\gamma}{\sqrt{\omega}} \sin \left( \frac{|x|-\pi/2}{\sqrt{\omega}}\right)= \lambda\, \uog\, .
$$
Analogously, for $(x,y)\in \Omega^\omega_\gamma$ and  $|y|\geq \frac{\pi}{2}$, we have
$$
 D^2\uog (x,y)= \left( 
 \begin{array}{cc}
- \gamma\, \cos x & 0\\
 0 & \frac{1}{\sqrt{\omega}} \sin \left( \frac{|y|-\pi/2}{\sqrt{\omega}}\right)
 \end{array}\right)
 $$
and, since $|x|\leq \frac{\pi}{2}$ and $0\leq \frac{|y|-\pi/2}{\sqrt{\omega}}<\arcsin \left( \frac{\gamma}{ \sqrt{\omega}}\right) \leq \frac{\pi}{2}$,  we again conclude
$$
-\Mp \left( D^2\uog\right) =\lambda \, \gamma\, \cos x- \frac{\Lambda}{\sqrt{\omega}} \sin \left( \frac{|y|-\pi/2}{\sqrt{\omega}}\right)= \lambda\, \uog\, .
$$ \end{proof}
 
 \begin{remark}
  {\rm 
 Let us remark that for  $\omega=1$ the only admissible value for $\gamma$ is $\gamma=1$ and there is only one set $\Omega^1_1$.
 In this case, up to a rotation, $\Omega^1_1$ is the square $\{ |x|< \pi/\sqrt{2}\, ,\ |y|<\pi/\sqrt{2}\}$ and $u^1_1(x,y)=\cos \left( \frac{x}{\sqrt{2}}\right)\, \cos \left( \frac{y}{\sqrt{2}}\right)$ is the first eigenfunction of the Laplace operator, associated with the first eigenvalue $\lambda_1 = \mu \left( -\Delta, \Omega^1_1\right)=1$.
 
 For  $\omega>1$, we have identified the family of  bounded domains $\Oog$, $\frac{1}{\sqrt{\omega}}\leq  \gamma\leq \sqrt{\omega}$,  in all of which the positive principal eigenvalue of $\Mp$ is $\lambda$. Note that $\Oog$ is a smooth set except for $\gamma=\sqrt{\omega}$ and the symmetric case  $\gamma=1/\sqrt{\omega}$. $\partial \Omega^\omega_{\sqrt{\omega}}$ has  singularity  points  at $ \left(0, \pm (1+\sqrt{\omega}) \frac{\pi}{2}\right)$, where an angle of amplitude  $2\arctan \left(\frac{1}{\sqrt{\omega}}\right)$ occurs (see Figure 1). Moreover, for $(x,y)\in \Omega^\omega_{\sqrt{\omega}}\cap \{ |y|>\pi/2\}$, the eigenfunction $u^\omega_{\sqrt{\omega}}$ has the expression
 $$
 u^\omega_{\sqrt{\omega}}(x,y)=2\cos \left( \frac{ \frac{|y|-\pi/2}{\sqrt{\omega}}+\frac{\pi}{2}+x}{2}\right)\, \cos \left( \frac{ \frac{|y|-\pi/2}{\sqrt{\omega}}+\frac{\pi}{2}-x}{2}\right)
 $$
 showing that $u^\omega_{\sqrt{\omega}}(x,y)$ vanishes quadratically as $\Omega^\omega_{\sqrt{\omega}}\ni (x,y)\to \left( 0,\pm (1+\sqrt{\omega}) \frac{\pi}{2}\right)$. This property is consistent with the fact that the homogeneous problem
 \begin{equation}\label{omo}
 \left\{ \begin{array}{c}
 \Mp(D^2\Phi) =0\qquad \hbox{in } \mathcal{C}\\[1ex]
 \Phi =0 \qquad \hbox{on } \partial \mathcal{C}
 \end{array}\right.
 \end{equation}
 where $\mathcal{C}$ is the plane cone  $\mathcal{C}=\{ y> \sqrt{\omega}|x|\}$, has the  positive, degree 2 homogeneous solution $\Phi (x,y)=y^2-\omega\, x^2$ (see \cite{L}).  Indeed, by the comparison principle, it immediately follows that
 $$
\liminf_{\Omega^\omega_{\sqrt{\omega}}\ni (x,y)\to \left(0, \pm (1+\sqrt{\omega}) \frac{\pi}{2}\right)} \frac{u^\omega_{\sqrt{\omega}}(x,y)}
{\Phi\left( x,(1+\sqrt{\omega})\frac{\pi}{2}\mp y\right)} >0\, .
$$ 
}
 \end{remark}
\bigskip

\begin{remark}{\rm  The function $u^\omega_\gamma$  can be extended in order to obtain a changing sign eigenfunction for $\Mp$ in the whole $\R^2$.
Precisely, for any $\gamma>0$, let us define in the square $\left\{ |x|\, ,\ |y| \leq (1+\sqrt{\omega})\frac{\pi}{2}\right\}$
$$
\uog (x,y)=\left\{
\begin{array}{ll}
 \gamma\,  \cos x +\cos y &  \hbox{if } |x|\,,\  |y|\leq \frac{\pi}{2}\\[2ex]
-\gamma\, \sqrt{\omega} \sin \left( \frac{|x|-\pi/2}{\sqrt{\omega}}\right) +\cos y   &  \hbox{if }    \frac{\pi}{2}< |x|\leq (1+\sqrt{\omega})\frac{\pi}{2}\, ,\ |y|\leq \frac{\pi}{2}\\[3ex]
\gamma\,  \cos x  -\sqrt{\omega} \sin \left( \frac{|y|-\pi/2}{\sqrt{\omega}} \right)  &  \hbox{if }    |x| \leq \frac{\pi}{2}\, ,\ \frac{\pi}{2}< |y|\leq (1+\sqrt{\omega})\frac{\pi}{2}\\[3ex]
-\sqrt{\omega}\left( \gamma  \sin \left( \frac{|x|-\pi/2}{\sqrt{\omega}}\right) +\sin \left( \frac{|y|-\pi/2}{\sqrt{\omega}} \right)\right) &  \hbox{if }  \frac{\pi}{2}< |x|\,,\  |y| \leq (1+\sqrt{\omega})\frac{\pi}{2} 
\end{array} \right.
 $$
and extend $\uog$ periodically both with respect to $x$ and $y$. Then, by arguing as in Theorem \ref{simm}, it is easy to see that 
$$
\Mp (D^2\uog )+\lambda\, u =0 \qquad \hbox{in } \R^2\, .
$$
The set where $\uog$ is positive has bounded connected components if and only if $\frac{1}{\sqrt{\omega}}\leq \gamma \leq \sqrt{\omega}$, and in this case they are nothing but  translations of $\Oog$. Conversely, the connected components of the set 
$D^\omega_\gamma =\{ \uog<0\}$ are unbounded for any $\gamma >0$. For $\frac{1}{\sqrt{\omega}}< \gamma < \sqrt{\omega}$ $D^\omega_\gamma$ is connected and unbounded in both $x$ and $y$ direction, whereas either for $\gamma \leq \frac{1}{\sqrt{\omega}}$ or for $\gamma\geq \sqrt{\omega}$ the connected components of $D^\omega_\gamma$ are contained in unbounded respectively horizontal or vertical stripes, see Figure 2. Since $\uog$ is a negative eigenfunction for $\Mp$ in each of the connected components of $D^\omega_\gamma$, we can say that for these sets one has $\mu^-=\lambda$. We finally remark that this construction does not yield a changing sign eigenfunction for a bounded domain, so that we cannot calculate eigenvalues different from the principal ones.}
 \end{remark}
 
 \begin{figure}
 \includegraphics[height=40mm]{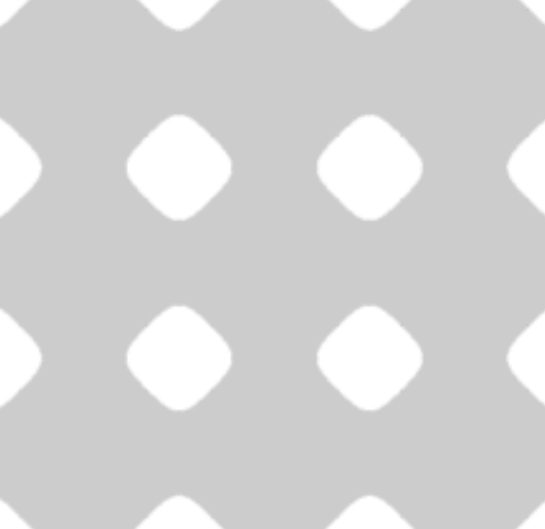} \includegraphics[height=40mm]{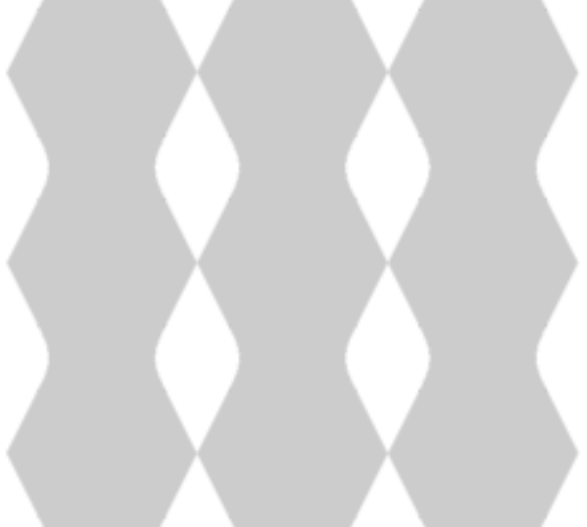}\includegraphics[height=40mm]{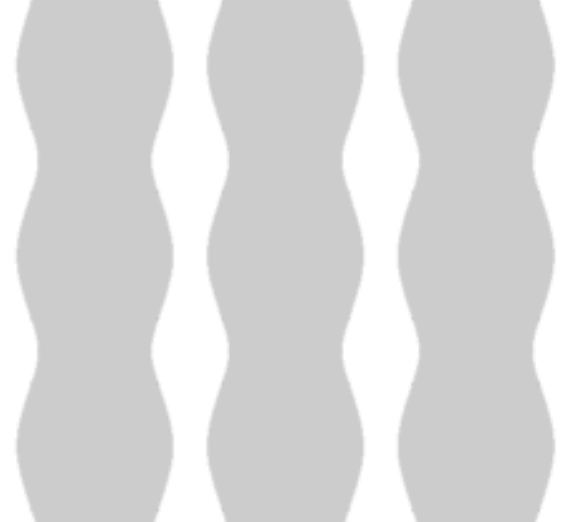}
\caption{${ D}^{\omega}_1$,  ${ D}^{\omega}_{\sqrt{\omega}}$ and ${ D}^{\omega}_{\gamma}$ for $\gamma >\sqrt{\omega}$.}
\end{figure}
\bigskip\bigskip

\noi   Let us now enlarge, by deforming the sets $\Oog$,  the class of domains for which we can evaluate the positive principal eigenvalue of $\Mp$. For any $a\in \R$ with $|a|<\pi$ let us consider the non singular matrix
  $$
  C_a= \left( 
 \begin{array}{cc}
\sqrt{1 -\left(\frac{a}{\pi}\right)^2}  & 0\\[2ex]
 \frac{a}{\pi} & 1
 \end{array}\right)
 $$
and let us denote by $C_a:\R^2\to \R^2$ also the linear transformation induced by $C_a$. We observe that $C_a$ maps the square $Q=\{ |x|+|y|<\pi\}$ with side $\sqrt{2}\pi$ into the rectangle $R=\left\{ |x| +\left|  \frac{\sqrt{\pi^2-a^2}y-a\,x}{\pi}\right|<\sqrt{\pi^2-a^2}\right\}$ with sides $\sqrt{2\pi(\pi-a)}$ and $\sqrt{2\pi(\pi+a)}$, and the square $\{ |x|\,,\ |y|<\pi/2\}$ onto the rhombus $\left\{ |x|< \frac{\sqrt{\pi^2-a^2}}{2}\, ,\ \left| y-\frac{a}{\sqrt{\pi^2-a^2}}x\right| <\frac{\pi}{2}\right\}$. Let us further set
$$
\Ooga \, := C_a \, \left( \Oog\right)
$$
and
$$
\uoga (x,y) \, : = \uog \left( C_a^{-1} (x,y)\right) \, , \quad (x,y)\in \Ooga\, ,
$$
where $\uog$ is defined in Theorem \ref{simm}, see Figure 3.
\begin{figure}
\includegraphics[height=40mm]{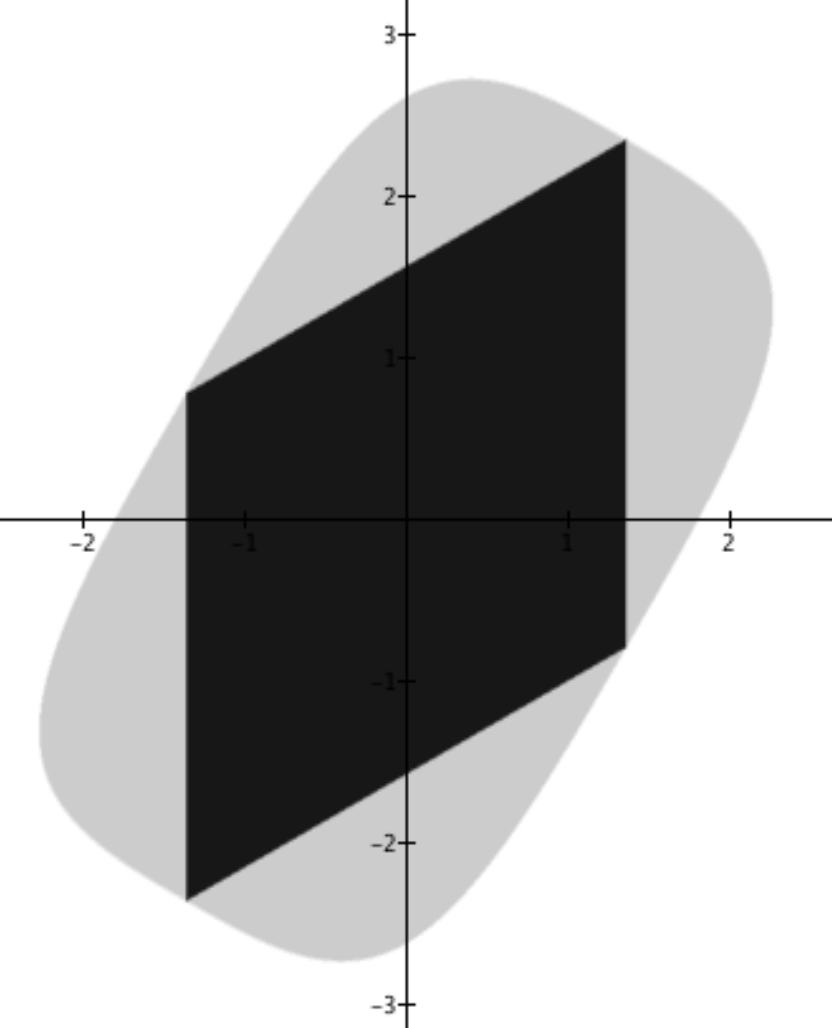}

\caption{The domain $\Ooga$; in the black part $\uoga$ is concave.}
\end{figure}
\begin{theorem}\label{nonsimm} Given $\Lambda\geq \lambda>0$ let us set $\omega =\frac{\Lambda}{\lambda}\geq1$. Then, for any $\frac{1}{\sqrt{\omega}}\leq \gamma\leq \sqrt{\omega}$ and $|a|<\pi$ the function $\uoga$ satisfies
\begin{equation}\label{super}
\left\{ \begin{array}{l}
-\Mp (D^2 \uoga ) \geq \frac{\lambda \, \pi^2}{\pi^2-a^2} \uoga \quad \hbox{in } \Ooga\\[2ex]
\uoga >0\ \hbox{in } \Ooga\, ,\ \uoga =0\ \hbox{on } \partial \Ooga
\end{array}\right.
\end{equation}
As a consequence, the positive principal eigenvalue of $\Mp$ in $\Ooga$ satisfies
\begin{equation}\label{lb}
\mu \left( \Ooga \right) \geq \frac{\lambda\, \pi^2}{\pi^2-a^2}\, ,
\end{equation}
and equality holds if and only if either $\omega=1$ or $a=0$.
 \end{theorem}
 
 \begin{proof}
 Let us compute. We  have
 $$
 D^2\uoga (x,y)= \left( C_a^{-1}\right)^t D^2\uog \left( C_a^{-1} (x,y)\right) C_a^{-1}\, ,
 $$
 with
 $$
 C_a^{-1}=\left( 
 \begin{array}{cc}
\frac{\pi}{\sqrt{\pi^2-a^2}}  & 0\\[2ex]
- \frac{a}{\sqrt{\pi^2-a^2}} & 1
 \end{array}\right)\, .
 $$
Since $D^2\uog$ is diagonal, by setting 
$$
\left\{ \begin{array}{l}
X=\frac{\pi}{\sqrt{\pi^2-a^2}}  x\\[2ex]
Y=y- \frac{a}{\sqrt{\pi^2-a^2}} x
\end{array}
\right.
$$
we then obtain
$$
D^2\uoga (x,y)=\left(
\begin{array}{cc}
\frac{\pi^2}{\pi^2-a^2} (\uog)_{xx}(X,Y) +\frac{a^2}{\pi^2-a^2}(\uog)_{yy}(X,Y) & -\frac{a}{\sqrt{\pi^2-a^2}}(\uog)_{yy}(X,Y)\\[2ex]
-\frac{a}{\sqrt{\pi^2-a^2}}(\uog)_{yy}(X,Y) & (\uog)_{yy}(X,Y)
\end{array} \right)\, .
$$
Note that, in particular,
$$
{\rm det}(D^2\uoga (x,y))= \frac{\pi^2}{\pi^2-a^2}{\rm det}(D^2\uog(X,Y))\, .
$$
Therefore, for $(x,y)\in C_a\left(  \left\{|X|\, ,\ |Y|\leq \frac{\pi}{2}\right\} \right)$, $\uoga (x,y)$ is concave like $\uog(X,Y)$ and it follows that
$$
-\Mp(D^2\uoga )=-\lambda\, \Delta \uoga =- \frac{\lambda\, \pi^2}{\pi^2-a^2} \Delta \uog= \frac{\lambda\, \pi^2}{\pi^2-a^2} \uog =\frac{\lambda\, \pi^2}{\pi^2-a^2} \uoga\, .
$$
Otherwise, for $(x,y)\in \Ooga$ such that either $|X|>\frac{\pi}{2}$ or $|Y|>\frac{\pi}{2}$, we have ${\rm det}(D^2\uoga (x,y))<0$, and, by computing the  eigenvalues of $D^2\uoga$ and recalling the expressions of $(\uoga)_{xx}$ and $(\uoga)_{yy}$ from the proof of Theorem \ref{simm}, we get
$$
\begin{array}{ll}
-\Mp (D^2\uoga ) & = -\frac{\lambda\, \pi^2}{2(\pi^2-a^2)} \left[ (\omega+1) \left( (\uog)_{xx}+(\uog)_{yy}\right) \right.\\[2ex]
& \qquad \qquad \quad \left. +(\omega -1) \sqrt{(\uog)_{xx}^2
+(\uog)_{yy}^2+2\left( \frac{2a^2}{\pi^2}-1\right) (\uog)_{xx}(\uog)_{yy}}\right]\\[2ex]
& \geq  -\frac{\lambda\, \pi^2}{2(\pi^2-a^2)} \left[ (\omega+1) \left( (\uog)_{xx}+(\uog)_{yy}\right) +(\omega -1) \left| (\uog)_{xx}-(\uog)_{yy}\right|\right]\\[2ex]
& = \frac{\lambda\, \pi^2}{(\pi^2-a^2)} \, \uoga\, ,
\end{array}
$$
and equality holds in the above if and only if either $\omega=1$ or $a=0$. Therefore, 
 $\uoga$ satisfies \refe{super}, and \refe{lb} follows immediately from the definition of the positive principal eigenvalue for  $\Mp$. Moreover, equality holds in \refe{lb} if and only if $\uoga$ is the principal eigenfunction for $\Mp$ in $\Ooga$, see Corollary 2.1 in \cite{BNV} or Theorem 4.4 in \cite{P}. Hence, equality holds in \refe{lb} if and only if either $\omega=1$ or $a=0$.
\end{proof}

As a consequence of Theorems \ref{simm} and \ref{nonsimm}, we can deduce that, among all sets $\Ooga$ and their rescaled $\delta \,\Ooga$ with $\delta>0$, for equal area the minimum of the principal eigenvalue for $\Mp$ is achieved on the most symmetric domain, that is some rescaled of $\Omega^\omega_1$. We will denote by $|\Omega|$ the area (two dimensional Lebesgue measure) of any set $\Omega\in \R^2$, and by $\mu (\Omega)$  the positive principal eigenvalue of $\Mp$ in the domain $\Omega$.

\begin{corollary}
Given $\Lambda\geq \lambda>0$,  let us set $\omega =\frac{\Lambda}{\lambda}\geq1$. Then
$$
\mu \left( \frac{\Omega^\omega_1}{\sqrt{\left| \Omega^\omega_1\right|}}\right)=\min \left\{ \mu \left( \frac{\Ooga}{\sqrt{\left| \Ooga\right|}}\right)\, :\ \frac{1}{\sqrt{\omega}}\leq \gamma\leq \sqrt{\omega}\, ,\ |a|<\pi\right\}\, .
$$
\end{corollary}

\begin{proof}
By the homogeneity of the principal eigenvalue and by Theorem \ref{nonsimm}, we have
$$
\mu \left( \frac{\Ooga}{\sqrt{\left| \Ooga\right|}}\right) = \left| \Ooga\right| \, \mu \left( \Ooga\right) \geq \frac{\lambda\, \pi^2}{\pi^2-a^2} \, \left| \Ooga\right| \, .
$$
Moreover, one has
$$
\left| \Ooga\right| =\left| C_a\left( \Oog \right)\right| =\left| {\rm det}\left( C_a\right)\right|\, \left| \Oog\right| = \frac{\sqrt{\pi^2-a^2}}{\pi}\, \left| \Oog\right| \, ,
$$
so that
$$
\mu \left( \frac{\Ooga}{\sqrt{\left| \Ooga\right|}}\right) \geq \frac{\lambda\, \pi}{\sqrt{\pi^2-a^2}}\, \left| \Oog\right|\geq \lambda\, \left| \Oog\right|\, .
$$
On the other hand, by the definition of $\Oog$, we get
$$
\left| \Oog\right| = \pi^2 +4\sqrt{\omega} \int_{0}^{\pi/2} \left[ \arcsin \left( \frac{\gamma}{\sqrt{\omega}}\cos x\right) +\arcsin \left( \frac{1}{\gamma\, \sqrt{\omega}}\cos x\right) \right]\, dx\, ;
$$
hence,
$$
\frac{d}{d \gamma} \left| \Oog\right|= \frac{4\sqrt{\omega}}{\gamma} \int_0^{\pi/2} \left[ \frac{1}{\sqrt{\frac{\omega}{\gamma^2}-\cos^2 x}}-
\frac{1}{\sqrt{\omega\, \gamma^2-\cos^2 x}}\right]\, \cos x\, dx \left\{ \begin{array}{l}
\geq 0 \quad  \hbox{for } \gamma \geq 1\\[2ex]
\leq 0 \quad  \hbox{for } \gamma \leq 1
\end{array} \right.
$$
which shows that $\left| \Oog\right|$ is minimal for $\gamma=1$. In conclusion, by using also Theorem \ref{simm}, we deduce
$$
\mu \left( \frac{\Ooga}{\sqrt{\left| \Ooga\right|}}\right) \geq \lambda\, \left| \Oog\right|\geq \lambda\, \left| \Omega^\omega_1\right| = \mu \left( \frac{\Omega^\omega_1}{\sqrt{\Omega^\omega_1}}\right)\, .
$$
\end{proof}

\end{document}